\documentclass[11pt]{article}

\usepackage[utf8]{inputenc}
\usepackage[T1]{fontenc}
\usepackage[francais,english]{babel}

\usepackage{mathrsfs}
\usepackage{amsmath}
\usepackage{amssymb}
\usepackage{amsthm}

\usepackage{enumitem}

\usepackage{esvect}
  \usepackage[all]{xy} 

\usepackage{graphicx}

\usepackage{hyperref}

\newcommand{\ensemblenombre}[1]{\mathbb{#1}}
\newcommand{\N}{\ensemblenombre{N}}
\newcommand{\Z}{\ensemblenombre{Z}}

\newcommand{\R}{\ensemblenombre{R}}

\newcommand{\ab}{\mathrm{ab}}
\newcommand{\imb}{\mathrm{imb}}
\newcommand{\dist}{\mathrm{dist}}

\newcommand{\Fcal}{\mathcal{F}}
\newcommand{\Gcal}{\mathcal{G}}
\newcommand{\Hcal}{\mathcal{H}}

\newcommand{\Rcal}{\mathcal{R}}

\newcommand{\Afrak}{\mathfrak{A}}

\newtheoremstyle{perso}
   {4ex}
   {4ex}
   {\itshape}
   {-2ex}
   {\bfseries}
   {}
   {\newline}
   {}

\newtheoremstyle{perso2}
   {4ex}
   {4ex}
   {}
   {-2ex}
   {\bfseries}
   {}
   {\newline}
   {}

\newcounter{th}
\newtheorem{theorem}[th]{Theorem}
\newtheorem{lemma}{Lemma}
\newtheorem{corollary}{Corollary}
\newtheorem{proposition}{Proposition}
\newtheorem{remark}{Remark}

\newtheorem{definition}{Definition}

\newtheorem{theorembis}{Theorem}

\newtheorem{theoremter}{Theorem}

\newcounter{thfr}
\newtheorem{theoreme}[thfr]{Théorème}

\newtheorem{theoremebis}{Théorème}

\newtheorem{theoremeter}{Théorème}

\usepackage{geometry}
\title{A Rauzy fractal unbounded in all directions of the plane}
\author{M\'elodie Andrieu}
\date{}

\begin{document}

\maketitle

\begin{abstract}
	We construct an Arnoux-Rauzy word for which the set of all differences of two abelianized factors is equal to $\Z^3$. In particular, the imbalance of this word is infinite - and its Rauzy fractal is unbounded in all directions of the plane.
\end{abstract}

\selectlanguage{french}

\begin{abstract}
	Nous construisons explicitement un mot d'Arnoux-Rauzy pour lequel l'ensemble des différences possibles des facteurs abélianisés est égal à $\Z^3$. En particulier, le déséquilibre de ce mot est infini, et son fractal de Rauzy n'est borné dans aucune direction du plan.   
\end{abstract}

%
%

\section{Introduction} \label{sect:introduction}

À l'algorithme de fraction continue soustractif décrit par l'itération de l'application (dite de Farey)
$$
\begin{array}{llll}
(\R^+)^2 & \rightarrow & (\R^+)^2 & \\
(x,y) & \mapsto & (x-y,y) & \qquad \text{si $x \geq y$,} \\
& & (x,y-x) & \qquad \text{sinon,}
\end{array}
$$
 est associée une classe particulière de mots infinis binaires appelés mots sturmiens. Rappelons qu'un \emph{mot} est une suite finie ou infinie d'éléments (\emph{lettres}) pris dans un ensemble fini (\emph{alphabet}). Les mots sturmiens jouissent de nombreuses caractérisations combinatoires, arithmétiques et géométriques (consulter \cite{Loth} pour une introduction générale). En particulier, ce sont exactement les mots apériodiques binaires dont le déséquilibre vaut 1, c'est-à-dire dans lesquels 
 tous les facteurs de même longueur (un \emph{facteur} de longueur $n$  est un sous-mot constitué de $n$ lettres consécutives) contiennent, à 1 près, le même nombre de 0 (et donc, à 1 près également, le même nombre de 1). Par exemple, un mot commençant par $w=001000100100010001001…$ pourrait être sturmien, tandis qu'un mot commençant par $w=011011100...$ ne l'est pas, car il contient les facteurs $11$ et $00$.
 Cette propriété garantit en particulier que les lettres $0$ et $1$ sont uniformément distribuées par rapport à une mesure de probabilité $\nu$ sur $\{0,1\}$, et que l'écart entre la somme de Birkhoff $1/N\sum_{n=0}^{N-1} 1\!\!1_{\{0\}}(w[n])$, qui mesure la fréquence de $0$ observée parmi les $N$ premières lettres du mot $w$, et sa valeur attendue $\nu(0)$ (appelée fréquence de $0$) est majoré par $1/N$. D'un point de vue géométrique, cela signifie que les points  $P_N:=\sum_{n=0}^{N}e_{w[n]}$, où $(e_0, e_1)$ désigne la base canonique de $\R^2$, restent à une distance bornée de la droite portée par le vecteur fréquence $(\nu(0),\nu(1))$. On appelle \emph{ligne brisée} associée à $w$ la suite $(P_N)_{N\in \N}$.
 En informatique, les lignes brisées associées aux mots sturmiens sont utilisées pour discrétiser les droites de pentes irrationnelles.
 
 Depuis Jacobi, plusieurs algorithmes ont été proposés pour généraliser les fractions continues à des triplets de réels positifs (on peut consulter à ce sujet le livre \cite{Schw00}). De tels algorithmes devraient permettre d'approcher simultanément et efficacement deux réels par une suite de couples de nombres rationnels.

 Dans ce document, nous nous intéressons aux mots d'Arnoux-Rauzy, introduits par Arnoux et Rauzy dans \cite{AR91}, qui sont les mots ternaires associés à l'algorithme (défini sur un ensemble de mesure nulle) : 
 $$ \label{FAR}
 \begin{array}{lllll}
 F_{AR} : \;&(\R^+)^3 & \rightarrow & (\R^+)^3 & \\
 &(x,y,z) & \mapsto & (x-y-z,y,z) & \qquad \text{si $x \geq y+z$,} \\
 && & (x,y-x-z,z) & \qquad \text{si $y \geq x+z$,}\\
 & & & (x,y,z-x-y) & \qquad \text{si $z \geq x+y$.}
 \end{array}
 $$
 
 Parce qu'ils conservent de nombreuses propriétés combinatoires des mots sturmiens, les mots d'Arnoux-Rauzy sont souvent présentés comme leur généralisation. En particulier, on peut montrer qu'ils admettent un vecteur fréquence des lettres.
 Aussi, une façon d'étudier la ligne brisée (tridimensionnelle) associée à un mot d'Arnoux-Rauzy consiste à la projeter, parallèlement au vecteur fréquence, sur le plan diagonal $\Delta_0: x+y+z=0$. On appelle \emph{fractal de Rauzy de $w$} l'adhérence de cet ensemble de points. 
 
 Jusqu'en 2000, on a pensé que, comme pour les mots sturmiens, le déséquilibre des mots d'Arnoux-Rauzy était borné, ou au moins fini. Cassaigne, Ferenczi et Zamboni \cite{CFZ01} ont contredit cette conjecture en construisant un mot d'Arnoux-Rauzy de déséquilibre infini - un mot donc, dont la ligne brisée s'écarte régulièrement et de plus en plus loin de sa direction moyenne, ou, dit encore autrement, un mot dont le fractal de Rauzy n'est pas borné.
 
 Aujourd’hui, on ne sait presque rien sur les propriétés géométriques et topologiques de ces fractals de Rauzy déséquilibrés. Le théorème d'Oseledets \cite{Ose68} suggère toutefois que ces fractals sont contenus dans une bande du plan ; en effet, si les exposants de Lyapounov associés au produit de matrices donné par l'algorithme existent, l'un de ces exposants au moins doit être négatif puisque leur somme est nulle.

Dans cette note, nous prouvons que cette intuition est fausse.
 
 \begin{theoreme}\label{th:main_french}
 	Il existe un mot d'Arnoux-Rauzy dont le fractal de Rauzy n'est borné dans aucune direction du plan.
 \end{theoreme}
 
 La construction que nous présentons s'adapte à la classe des mots associée à l'algorithme de fraction continue multidimensionnelle de Cassaigne-Selmer, introduite dans \cite{CLL17}, ainsi qu'aux mots épisturmiens stricts, qui sont la généralisation des mots d'Arnoux-Rauzy. Rappelons qu'un mot sur un alphabet contenant $d$ lettres est \emph{épisturmien strict} si son langage est clos par miroir et s'il admet, pour chaque longueur $n$, un unique facteur multi-prolongeable à droite, et si ce facteur peut de plus être prolongé par chacune des $d$ lettres de l'alphabet.

 \begin{theoremebis}\label{th:Cadic_french}
 	Il existe un mot C-adique $w_{\infty}$ dont le fractal de Rauzy n'est borné dans aucune direction du plan.
 \end{theoremebis}
 
 \begin{theoremeter}\label{th:episturmian_french}
 	Soit $d \geq 3$. Il existe un mot episturmien strict $w_{\infty}$ sur l'alphabet $\{1,...,d\}$ tel que pour tout hyperplan $\Hcal$ de $\R^d$, la distance des points de la ligne brisée  $(\ab(p_n(w_{\infty})))_{n \in \N}$ à l'hyperplan $\Hcal$ n'est pas bornée.
 \end{theoremeter}
 
Les démonstrations des Théorèmes \ref{th:Cadic_french} and \ref{th:episturmian_french} reposent sur des techniques similaires à celles du Théorème \ref{th:main_french}, et sont intégralement rédigées dans \cite{And20}.

 Par ailleurs, nous proposons une preuve élémentaire du :
 
 \begin{theoreme}\label{th:frequencies_french}
Le vecteur fréquence des lettres d'un mot d'Arnoux-Rauzy a des coordonnées rationnellement indépendantes.
 \end{theoreme} 
 Ce résultat, conjecturé par Arnoux et Starosta en 2013 \cite{AS13}, a été démontré très récemment par des moyens plus sophistiqués par Dynnikov, Hubert et Skripchenko \cite{DPS}.

   Le Théorème \ref{th:frequencies_french}  est en fait vrai en toute dimension (voir \cite{And20} pour une preuve complète):
   
   \begin{theoremebis}
   	Soit $d \geq 2$. Le vecteur fréquence des lettres d'un mot épisturmien strict sur l'alphabet $\{1,...,d\}$ a des coordonnées rationnellement indépendantes.
   \end{theoremebis}
   
 %
 %

 \selectlanguage{english}
 \section{Introduction (short English version)}
 
 Until 2000, it was believed that, as for Sturmian words, the imbalance of Arnoux-Rauzy words was bounded - or at least finite. Cassaigne, Ferenczi and Zamboni disproved this conjecture by constructing an Arnoux-Rauzy word with infinite imbalance, i.e. a word whose broken line deviates regularly and further and further from its average direction \cite{CFZ01}. Today, we know virtually nothing about the geometrical and topological properties of these unbalanced Rauzy fractals. The Oseledets theorem suggests that these fractals are contained in a strip of the plane: indeed, if the Lyapunov exponents of the matricial product associated with the word exist, one of these exponents at least is nonpositive since their sum equals zero. This article aims at disproving this belief.
 
 \begin{theorem}\label{th:th1}
 	There exists an Arnoux-Rauzy word whose Rauzy fractal is unbounded in all directions of the plane.
 \end{theorem}
 
 Theorem \ref{th:th1} also holds, on one hand, for C-adic words, which are the infinite words over $\{1,2,3\}$ associated with the Cassaigne-Selmer multidimensional continued fraction algorithm introduced in \cite{CLL17} and, on the other hand, for strict episturmian words, which are the generalization of Arnoux-Rauzy words. We recall that a strict episturmian word is a word whose language is close by mirror and which admits, for each length, a unique right-special factor -which is moreover prolonged by each letter in the alphabet.
  
  \begin{theorembis}\label{th:Cadic}
  	There exists $w_{\infty}$ a C-adic word whose Rauzy fractal is unbounded is all directions of the plane.
  \end{theorembis}
  
 \begin{theoremter}\label{th:episturmian}
 	Let $d \geq 3$. There exists $w_{\infty}$ a strict episturmian word  over the alphabet $\{1,...,d\}$ such that for any hyperplane $\Hcal$ in $\R^d$, the distance between $\Hcal$ and the broken line $(\ab(p_n(w_{\infty})))_{n \in \N}$ is unbounded. 
 \end{theoremter}
 
 The proofs of Theorems \ref{th:Cadic} and \ref{th:episturmian} are based on techniques similar to those of Theorem \ref{th:main}; they can be found in \cite{And20}.

 Besides, we propose an elementary proof of:
 \begin{theorem}\label{th:th2}
 	The vector of letter frequencies of any Arnoux-Rauzy word has rationally independent entries.
 \end{theorem}
 
 This theorem completes the works of Arnoux and Starosta, who conjectured it in 2013, to prove that the Arnoux-Rauzy continued fraction algorithm  detects all kind of rational dependencies \cite{AS13}. Note that it has been recently proved by Dynnikov, Hubert and Skripchenko  using quadratic forms \cite{DPS}.
 
  Again, with a similar proof (see \cite{And20}), this result holds in arbitrary dimension:
 
 \begin{theorembis}
 Let $d \geq 2$. The vector of letter frequencies of any strict episturmian word over $\{1,...,d\}$ has
 	rationally independent entries.
 \end{theorembis}

\setcounter{th}{0}

 \section{Preliminaries}
 \label{sect:preliminaries}


We denote by $\Afrak^*$ the set of all finite words over an alphabet $\Afrak$.
A finite word  $u = u[0]u[1] ... u[n-1]$, where $u[k]$ denotes the $(k+1)$-th letter of $u$, is a \emph{factor of length $n$} of a (finite or infinite) word $w$ if there exists a nonnegative integer $i$ such that for all $k \in \{0,...,n-1\}$, $w[i+k]=u[k]$; in the particular case $i = 0$, we say that $u$ is the \emph{prefix of length $n$} of $w$, and denote it by $u=p_n(w)$. We denote by $\Fcal_n(w)$ the set of factors of $w$ of length $n$ and by $\Fcal(w)$ its set of factors of all lengths.


A \emph{substitution} is an application mapping letters to finite words: $\Afrak \mapsto \Afrak^*$, that we extend into a morphism on the free monoid for the concatenation operation $\Afrak^*$ on one hand, and on the set of infinite words $\Afrak^{\N}$ on the other hand. Three substitutions will be of high interest in this paper: $\sigma_1$, $\sigma_2$ and $\sigma_3$ defined over $A=\{1,2,3\}$ by:
$$
\begin{array}{ll}
\sigma_i : & A \rightarrow A^* \\
& i \mapsto i \\
& j \mapsto ij \text{ for } j\in A\backslash \{i\}.
\end{array}
$$
They are called \emph{Arnoux-Rauzy substitutions}; we denote $AR=\{ \sigma_1,\sigma_2, \sigma_3\}$. The set $AR$ can be seen as a three letter alphabet -it should not be confused with $A = \{1,2,3\}$ over which the substitutions are defined. As much as we can, we refer to the elements of $AR^*$ or $AR^{\N}$ as "sequences" instead of "words"; nonetheless, some tools like the notions of factor and prefix will turn out to be useful for this second alphabet as well, especially in Section \ref{sect:results}. 

\medskip

The set $\Afrak^{\N}$ of infinite words over $\Afrak$ is endowed with the distance $\delta$: for all $w,w' \in \Afrak^{\N}$, $\delta(w,w') = 2^{-n_0}$, where $n_0=\min\{n\in \N \vert w[n]\neq w'[n]\}$ if $w \neq w'$, and $\delta(w,w')=0$ otherwise. We say that a sequence of finite words $(u_n)_{n \in \N} \in (\Afrak^*)^{\N}$ \emph{converges} to an infinite word $w \in \Afrak^{\N}$ if for any sequence of infinite words $(v_n)_{n \in \N} \in (\Afrak^{\N})^{\N}$, the sequence of infinite words $(u_n\cdot v_n)_{n \in \N} \in (\Afrak^{\N})^{\N}$ converges to $w$. 

If $(s_n)_{n \in \N} \in AR^{\N}$ is a sequence containing infinitely many occurrences of each Arnoux-Rauzy substitution $\sigma_1, \sigma_2$ and $\sigma_3$, then the sequence of finite words $(s_0 \circ ... \circ s_{n-1}(\alpha))$, with $\alpha \in A$, converges to an infinite word $w_0$ which does not depend on $\alpha$. The infinite words $w_0$ obtained this way are called \emph{standard Arnoux-Rauzy words}.
An infinite word $w$ is an \emph{Arnoux-Rauzy word} if it has the same set of factors than a standard Arnoux-Rauzy word $w_0$. One can show that the standard Arnoux-Rauzy word $w_0$ and the \emph{directive sequence} $(s_n)_{n\in\N}$ associated with $w$ are unique. This definition of Arnoux-Rauzy words is equivalent to the more usual one:  
an infinite word is an \emph{Arnoux-Rauzy word} if it has complexity $2n+1$ and admits exactly one right and one left special factor of each length.
\medskip

Given a finite word $u \in \Afrak^*$ and a letter $\alpha \in \Afrak$, we denote by $\vert u \vert_{\alpha}$ the number of occurrences of $\alpha$ in $u$. The \emph{abelianized vector} of $u$, sometimes called \emph{Parikh vector} of $u$, is the vector $\ab(u) = (\vert u \vert_{\alpha})_{\alpha \in \Afrak}$, which counts the number of times that each letter occurs in the finite word $u$. At this point, it is useful to order the alphabet.  
For the convenience of typing, we choose to represent abelianized words as line vectors.  Observe that the sum of the entries of $\ab(u)$ is equal to the \emph{length} of the word $u$, that we denote by $\vert u \vert$. Now, given a substitution $s : \Afrak \rightarrow \Afrak^*$, the \emph{incidence matrix} of $s$ is the matrix $M_s$ whose $i-th$ row is the abelianized of the image by $s$ of the $i-th$ letter in the alphabet.
For instance, the incidence matrices of the Arnoux-Rauzy substitutions are:
$$M_{\sigma_1} = \begin{pmatrix}
1 & 0 & 0\\
1 & 1 & 0 \\
1 & 0 & 1
\end{pmatrix}, \qquad M_{\sigma_2} = \begin{pmatrix}
1 & 1 & 0\\
0 & 1 & 0 \\
0 & 1 & 1
\end{pmatrix} \quad \text{ and } \quad M_{\sigma_3} = \begin{pmatrix}
1 & 0 & 1\\
0 & 1 & 1 \\
0 & 0 & 1
\end{pmatrix} \quad \in GL_3(\Z).
$$
Abelianized words and incidence matrices are made to satisfy: $\ab(s(u)) = \ab(u)M_s$ for any substitution $s : \Afrak \rightarrow \Afrak^*$ and any finite word $u \in \Afrak^*$. 


\medskip

If $w \in \Afrak^{\N}$ is an infinite word and $\alpha \in \Afrak$ is a letter, the frequency of $\alpha$ in $w$ is the limit, if it exists, of the proportion of $\alpha$ in the sequence of growing prefixes of $w$: $f_w(\alpha) = \lim_{n \rightarrow \infty} \frac{\vert p_n(w)\vert_\alpha}{n}$. We denote by $f_w = (f_w(\alpha))_{\alpha \in \Afrak}$ the \emph{vector of letter frequencies} of $w$, if it exists.
When the vector of letter frequencies exists, as it is the case for any Arnoux-Rauzy word, it is natural to study the difference between the predicted frequencies of letters and their observed occurrences. Given an infinite word $w \in \Afrak^{\N}$ for which the vector of letter frequencies is defined, we consider the \emph{discrepancy function}:
$$
\begin{array}{l}
\N \rightarrow \R\\
 n \mapsto \max_{\alpha\in \Afrak} \vert \, \vert p_n(w) \vert_\alpha - nf_w(\alpha)\vert .
\end{array}
$$ 
The discrepancy is linked to a combinatorial property: the imbalance. The {imbalance} of an infinite word $w$ is the quantity (possibly infinite) :
$$
\imb(w) \;=\; \underset{n \in \N}{\sup} \quad \underset{u,v \in \Fcal_n(w)}{\sup} \quad \vert \vert \ab(u)-\ab(v)\vert \vert_{\infty}.
$$
The imbalance of an infinite word $w$ is finite if and only if its discrepancy function is bounded. 
Geometrically, the discrepancy is linked to the diameter of the Rauzy fractal. 
Let $\Delta_0$ denotes the plane of $\R^3$ with equation $x+y+z=0$. For $w$ an Arnoux-Rauzy word, denote by $f_w$ its letter frequencies vector and by $\pi_w$ the (oblique) projection onto $\Delta_0$ associated with the direct sum: $\R f_w  \oplus \Delta_0 = \R^3$. The \emph{Rauzy fractal} of $w$, denoted by $\Rcal_w$, is the closure of the image of the set of abelianized prefixes of $w$ (the \emph{broken line of $w$}) by the projection $\pi_w$: $\Rcal_w = \overline{\cup_{k\in \N}\{\pi_w (\ab(p_k(w)))\}} \subset \Delta_0$. Note that the statement of our main result (Theorem \ref{th:main}) does not depend on the choice of the plane we project onto.

%
%

\section{Results} \label{sect:results}

\begin{lemma} \label{lemma:families}
	For any $(a,b,c) \in \Z^3$, there exists $s \in AR^*$ and there exist $u, v \in \Fcal(s(1))$ that satisfy  $\ab(u)-\ab(v)=(a,b,c)$.
\end{lemma}

\begin{remark}[Abuse of notation] \label{rk1}
	If $s = s_0\cdot ... \cdot s_{n-1} \in AR^*$, and if $w \in A^*\cup A^{\N}$, then $s(w)$ denotes the image of the word $w$ by the substitution $s_0 \circ ... \circ s_{n-1}$.
	\end{remark}
\begin{proof}
	Section \ref{sect:proof} is devoted to the proof of Lemma \ref{lemma:families}.
\end{proof}

Therefore, all standard Arnoux-Rauzy words -and thereby all Arnoux-Rauzy words- whose directive sequence starts with the prefix $s$ will admit $(a,b,c)$ as difference of abelianized factors.

\begin{lemma} \label{lemma:prefix}
	For any $p \in AR^*$ and any $(a,b,c) \in \Z^3$, there exists $s \in AR^*$ and there exist $u, v \in \Fcal(p\cdot s(1))$ that satisfy $\ab(u)-\ab(v)=(a,b,c)$.
\end{lemma}

\begin{proof}
	Let $p \in AR^*$ and $(a,b,c) \in \Z^3$. Denote by $M_p$ the incidence matrix of the substitution associated with $p$ (following Remark \ref{rk1}), which is a product of the Arnoux-Rauzy matrices $M_{\sigma_1},M_{\sigma_2}$ and $M_{\sigma_3}$, and thus belongs to $GL_3(\Z)$. By Lemma \ref{lemma:families}, there exists $s \in AR^*$ and there exist $u$ and $v \in \Fcal(s(1))$ such that $\ab(u)-\ab(v)= (a,b,c)\,M_p^{-1}$. But then, $p(u)$ and $p(v)$ are factors of $\Fcal(p\cdot s(1))$ and satisfy $\ab(p(u))-\ab(p(v))=(\ab(u)-\ab(v))M_p=(a,b,c)$. 
\end{proof}

We now construct a standard Arnoux-Rauzy word for which all triplets of integers can be obtained as a difference of two of its abelianized factors.

\begin{proposition} \label{prop:w_star}
	There exists an Arnoux-Rauzy word $w_{\infty}$ such that for all $(a,b,c) \in \Z^3$, there exist $u$ and $v \in \Fcal(w_{\infty})$ satisfying $\ab(u)-\ab(v)=(a,b,c)$.
\end{proposition}

\begin{proof}
	Let $\varphi : \N \rightarrow \Z^3$ a bijection (that can be chosen explicitly). We construct an infinite word $d \in AR^{\N}$ as the limit of the sequence of finite words $(p_k)_{k \in \N}\in (AR^*)^{\N}$ that we define by recurrence as follows.
	We first set $p_0$ as the prefix given by Lemma \ref{lemma:families} for $(a,b,c)=\varphi(0)$. Now, for $k \in \N$, we set $p_{k+1} = p_k.\sigma_1.\sigma_2.\sigma_3.s$, where $s \in AR^*$ is given by applying Lemma \ref{lemma:prefix} to the word $p_k.\sigma_1.\sigma_2.\sigma_3 \in AR^*$ and the vector $\varphi(k+1) \in \Z^3$. By construction, the sequence of finite words $(p_k)_{k \in \N}$ converges to an infinite sequence $d$ which contains infinitely many occurrences of $\sigma_1, \sigma_2$ and $\sigma_3$. This guarantees that the sequence of finite words $(d_0\circ ... \circ d_{n-1}(1))_{n\in \N}$ converges to an Arnoux-Rauzy word, that we denote by $w_{\infty}$. Finally, for any $k \in \N$, since the directive sequence of $w_{\infty}$ starts with the prefix $p_k$, there exist  $u_k, v_k \in \Fcal(w_{\infty})$ such that $\ab(u_k)-\ab(v_k) = \varphi(k)$.
	\end{proof}

\begin{corollary} \label{cor:imbalance} The imbalance of the word $w_{\infty}$ is infinite. 
\end{corollary}

\begin{proof}
	For any $n \in \N$, there exist $u_n$ and $v_n \in \Fcal(w_{\infty})$ such that $\ab(u_n)-\ab(v_n)=(n,0,-n)$; this implies both $\vert u_n\vert = \vert v_n \vert$ and $\vert u_n\vert_1 - \vert v_n \vert_1 = n$. The imbalance of $w_{\infty}$ is thus infinite.
\end{proof}

The imbalance of a word, which is a combinatorial quantity, is linked to the geometrical shape of its associated broken line. More precisely: a word $w$ admitting frequencies has an infinite imbalance if and only if its Rauzy fractal is unbounded.
We now propose to show that the word $w_{\infty}$ actually satisfies a stronger property: its Rauzy fractal is unbounded in \emph{all directions} of the plane. This relies on the following proposition.

\begin{proposition} \label{prop:geom}
	Let $w \in A^{\N}$. If for all $\textbf{d} \in \Z^3 \cap \Delta_0$, where $\Delta_0$ denotes the plane of $\R^3$ with equation $x+y+z=0$, there exist $u$ and $v \in \Fcal(w)$ such that $\ab(u)-\ab(v)=\textbf{d}$, then, for any plane $\Pi$ and for any $D \in \R^+$, there exists $k\in \N$ such that the euclidean distance between the point $\ab(p_k(w))$ and the plane $\Pi$ is larger than $D$.
\end{proposition}

\begin{proof}
	 	Without loss of generality, we can assume that $\Pi$ contains $(0,0,0)$. \\
	 	If $\Pi = \Delta_0$, then for any $D \in \R^+$, $\dist(\ab(p_k(w)),\Pi)> D$, with $k = \lfloor D\sqrt{3}/3 \rfloor +1$. \\
	 	Let $\Pi \neq \Delta_0$. By contradiction, assume that there exists $D \in \R^+$ such that for all nonnegative integer $k$, $\dist(\ab(p_k(w)),\Pi)\leq D$. Let $\textbf{d} \in \Z^3 \cap \Delta_0$ with $\dist(\textbf{d},\Pi)> 4D$, and factors $u, v \in \Fcal(w)$ such that $\ab(u)-\ab(v)=\textbf{d}$. Then, without loss of generality, we have $\dist(\ab(u),\Pi)> 2D$. Let $t \in A^*$ be such that $tu$ is a prefix of $w$. Then we have $\dist(\ab(t),\Pi)> D$ or $\dist(\ab(tu),\Pi)> D$, a contradiction.
\end{proof}

\begin{remark}
	Proposition \ref{prop:geom} and its proof remain valid by replacing $\Delta_0$ by any other plane whose intersection with $\Z^3$ is not trapped between two parallel lines.
\end{remark}

\begin{theorem}\label{th:main}
	There exists an Arnoux-Rauzy word whose Rauzy fractal is unbounded in all directions of the plane.
\end{theorem}

\begin{proof}
We obtain, by applying Proposition \ref{prop:geom} to the word $w_{\infty}$ described in Proposition \ref{prop:w_star} and to planes spanned by $f_w$ and a vector of $\Delta_0$, that the Rauzy fractal associated with $w_{\infty}$ cannot be trapped between two parallels lines.
\end{proof}

\section{Proof of Lemma \ref{lemma:families}}
\label{sect:proof}
	
	We consider the infinite oriented graph whose vertices are the elements of $\Z^3$ and whose edges map triplets to their images by one the 15 following applications. For $\delta \in \{-2,-1,0,1,2\}$ and $i \in \{1,2,3\}$, consider:
	
	$$\begin{array}{llllr}
	\tau_{i,\delta} : & \Z^3 & \rightarrow & \Z^3& \\
	& (x_j)_{j \in \{1,2,3\}} & \mapsto & (y_j)_{j \in \{1,2,3\}}  & \text{ where } y_i = x_1+x_2+x_3+\delta \text{ and } y_j = x_j \text{ for } j\neq i.
	\end{array}
	$$

	Our aim is to show that all vertices can be reached from the triplet $O=(0,0,0) \in \Z^3$, moving through a finite number of edges (see Definition \ref{def:accessible} and Proposition \ref{prop:tous} below.) The motivation lies in the following lemma. 
	
	\begin{lemma} \label{lemma:motiv}
		Let $d \in \Z^3$. If there exist $n \in \N$ and a finite sequence $(i_l,\delta_l)_{0\leq l\leq n-1} \in (\{1,2,3\}\times \{-2,-1,0,1,2\})^n$ such that $d = \tau_{i_{n-1},\delta_{n-1}} \circ ... \circ \tau_{i_0,\delta_0} (O)$, then there exist $s_1 \in AR^*$ and  $u, v  \in \Fcal(\sigma_{i_{n-1}}\circ ... \circ \sigma_{i_0}(s_1(1)))$ satisfying $\ab(u)-\ab(v)= d$.
	\end{lemma}

	\begin{proof}
		Let $d \in \Z^3$. Assume that there exist $n \in \N$ and $(i_l,\delta_l)_{0\leq l\leq n-1} \in (\{1,2,3\}\times \{-2,-1,0,1,2\})^n$ such that $d = \tau_{i_{n-1},\delta_{n-1}} \circ ... \circ \tau_{i_0,\delta_0} (O)$. 
		We are going to build iteratively two finite sequences of finite words $(u_l)$ and $(v_l)$, where $l \in \{0,...,n\}$, and $s_1 \in AR^*$, such that for all $l$, the words $u_l$ and $v_l$ are factors of $\sigma_{i_{l-1}}\circ ... \circ \sigma_{i_0}(s_1(1))$, and such that  $\ab(u_n)-\ab(v_n)=d$.
		
		First, we choose $s_1 \in AR^*$ that satisfies $\vert s_1(1) \vert \geq 2n$, and we set $u_0 = v_0 = p_n(s_1(1))$ (prefix of length $n$ of $s_1(1)$). Then, assuming that $u_l$ and $v_l \in \Fcal(\sigma_{i_{l-1}}\circ ... \circ \sigma_{i_0}(s_1(1)))$ are built, we set $\tilde{u}_{l+1}= \sigma_{i_l}(u_l)$ and $\tilde{v}_{l+1}= \sigma_{i_l}(v_l)$. From  $\tilde{u}_{l+1}$ and $\tilde{v}_{l+1}$, we define $u_{l+1}$ and $v_{l+1}$ according to the following table.
		\medskip
		
		\begin{center}
			\begin{tabular}{|l|ll|}
				\hline
				$\delta$   & choice for $u_{l+1}$ & choice for $v_{l+1}$ \\
				\hline
				$0$  & $\tilde{u}_{l+1}$ & $\tilde{v}_{l+1}$\\
				$1$  & $\tilde{u}_{l+1}.i_l$ & $\tilde{v}_{l+1}$\\
				$2$  & $\tilde{u}_{l+1}.i_l$ & $v_{l+1} \text{ such that } i_l.v_{l+1}=\tilde{v}_{l+1} (*)$\\
				$-1$  & $\tilde{u}_{l+1}$ & $\tilde{v}_{l+1}.i_l$\\
				$-2$  & $u_{l+1} \text{ such that } i_l.u_{l+1}=\tilde{u}_{l+1} (*)$ & $\tilde{v}_{l+1}.i_l$\\
				\hline
		\end{tabular}\end{center}
		\medskip
		
		\noindent We now justify that the steps marked with $(*)$ (removal of the initial $i_l$) are well-defined, and that $u_{l+1}$ and $v_{l+1}$ are, in all cases, factors of $\sigma_{i_l}\circ ... \circ\sigma_{i_0}(s_1(1))$.
		
		Observe that for any step $l \in \{0,...,n-1\}$, we remove at most one letter from the left and add at most one letter to the right of $\tilde{u}_{l+1}$ (resp. $\tilde{v}_{l+1}$). The Arnoux-Rauzy substitutions being nonerasing, we recursively check that (these properties hold symmetrically for $v_l$)
		:
		
		- the length of $u_l$ and its image $\tilde{u}_{l+1}$ is at least $n-l$; so we can always perform step $(*)$;
		
		- there is an occurrence of $u_l$ which is followed by at least $n-l$ letters in $\sigma_{i_{l-1}}\circ...\circ \sigma_{i_0}(s_1(1))$; so its image $\tilde{u}_{l+1}$ has also an occurrence in $\sigma_{i_{l}}\circ...\circ \sigma_{i_0}(s_1(1))$ which is followed by at least $n-l$ letters, and whose first following letter is $i_l$.

		Finally, in all  cases, the words $u_{l+1}$ and $v_{l+1}$ are factors of  $\sigma_{i_l}\circ ... \circ\sigma_{i_0}(s_1(1))$ and satisfy $\ab(u_{l+1})-\ab(v_{l+1}) = \tau_{i_l,\delta_l} (\ab(u_l)-\ab(v_l))$. In particular, at step $l=n-1$, the finite words $u_n$ and $v_n$ are factors of $\sigma_{i_{n-1}}\circ ... \circ \sigma_{i_0}(s_1(1))$ and satisfy $\ab(u_n)-\ab(v_n)= \tau_{i_{n-1},\delta_{n-1}} \circ ... \circ \tau_{i_0,\delta_0} (O)=d$. 
	\end{proof}

		In the sequel, it is convenient to introduce some vocabulary from graph theory.
		
		\begin{definition}\label{def:accessible}
			A triplet $(a,b,c) \in \Z^3$ is \emph{accessible} from a triplet $(d,e,f)$ if there exist a nonnegative integer $n$ and a finite sequence $(i_l,\delta_l)_{0\leq l\leq n-1} \in (\{1,2,3\}\times \{-2,-1,0,1,2\})^n$ such that $(a,b,c) = \tau_{i_{n-1},\delta_{n-1}} \circ ... \circ \tau_{i_0,\delta_0} ((d,e,f))$.
		\end{definition}
		
		\begin{proposition} \label{prop:tous}All triplets in $\Z^3$ are accessible from $O$.
		\end{proposition}
		
		The proof of Proposition \ref{prop:tous} lies on the two following lemmas.
		
		\begin{lemma} \label{lemma:symmetries}
			The triplet $(a,b,c) \in \Z^3$ is accessible from $O$ if and only if $(-a,-b,-c)$ is also accessible from $O$. Similarly, $(x_j)_{j \in \{1,2,3\}}$ is accessible from $O$ if and only if for all $s \in \mathfrak{S}_3$, where $\mathfrak{S}_3$ denotes the symmetric group acting on three elements, the triplet $(x_{s(j)})_{j \in \{1,2,3\}}$ is accessible from $O$.
		\end{lemma}
		
		\begin{proof}
			For the first assertion, change $\delta_l$ into $-\delta_l$  in the finite sequence of edges going from $O$ to $(a,b,c)$. For the second assertion, change $i_l$ into $s(i_l)$ in the finite sequence of edges going from $O$ to $(x_j)_{j \in \{1,2,3\}}$.
		\end{proof}
		
		\begin{lemma} \label{lemma:anchor}
			Let $a \in \N$. The triplet $(a,-a,-a) \in \Z^3$ is accessible from $O$.
		\end{lemma}
		
		\begin{proof}The lemma is trivially true for $a=0$. By recurrence, consider an arbitrary nonnegative integer $a$ such that the triplet $(a,-a,-a)$ is accessible from $O$. One can check that $(a+1,-a-1,a+1) = \tau_{1,1}\circ (\tau_{3,2})^{2a+1}\circ\tau_{2,-1} ((a,-a,-a))$. So the triplet $(a+1, -a-1,a+1)$ is accessible from $O$. But then, Lemma \ref{lemma:symmetries} indicates that $(a+1,-a-1,-a-1)$ is accessible from $O$.
		\end{proof}

		\begin{proof}[Proof of Proposition \ref{prop:tous}]
			The proof relies on the four following observations.
			\begin{itemize}
				\item The vertices $(a,b,-a)$ and $(a,-a,c)$ are accessible from $O$ for all $a,b,c \in \Z$. Indeed, it suffices to write $(a,b,-a)=(\tau_{2,0})^{a+b}((a,-a,-a))$ and $(a,-a,c)=(\tau_{3,0})^{a+c}((a,-a,-a))$ and remember that $(a,-a,-a)$ is accessible from $O$ by Lemma \ref{lemma:anchor}.
				\item The vertex $(a,c,c) = \tau_{2,0}((a,-a,c))$ is also accessible from $O$,
				\item If $a\geq b > c > -a$, then $(a,-a+b-c,c)=(\tau_{2,0})^{-1}(a,b,c)$ is closer to $(a,-a,-a)$ in sup norm, and we have $\vert -a+b-c \vert, \vert c \vert <a$,
				\item If $a\geq c > b >-a$, then $(a,b,-a-b+c)=(\tau_{3,0})^{-1}((a,b,c))$ is closer to $(a,-a,-a)$ than $(a,b,c)$, and we have $\vert b \vert, \vert -a-b+c \vert < a.$
			\end{itemize} 
		Let $a,b,c \in \Z^3$. By Lemma \ref{lemma:symmetries}, it suffices to deal with the case $\vert b \vert, \vert c \vert \leq \vert a \vert$, and $a >0$. Following the observations above, we recursively construct a finite sequence $(i_l)_{0\leq l \leq n-1} \in \{2,3\}^n$ such that $(a,b,c)=\tau_{i_{n-1},0}\circ ... \circ \tau_{i_0,0}((a,-a,-a))$. Since $(a,-a,-a)$ is accessible from $O$ (Lemma \ref{lemma:anchor}), the vertex $(a,b,c)$ is also accessible from $O$.\end{proof}

		\begin{proof}[Proof of Lemma \ref{lemma:families}]
			Lemma \ref{lemma:families} follows from Proposition \ref{prop:tous}, Definition \ref{def:accessible} and Lemma \ref{lemma:motiv}.
		\end{proof}

\begin{remark}
	The graph $\Gcal$ is a simplification, exploiting the remarkable properties of the substitutions $\sigma_1,\sigma_2$ and $\sigma_3$, of the imbalance automaton, introduced in \cite{And20} for a much wider range of S-adic systems (ie class of words obtained from a set of substitutions through \emph{directive sequences}).
\end{remark}

%
%

\section{The vector of letter frequencies of $w_{\infty}$ has rationally independent entries} \label{sect:entries}

We sketch an elementary proof of the much wider result:

\begin{theorem}\label{th:main2}
	The vector of letter frequencies of any Arnoux-Rauzy word has rationally independent entries.
\end{theorem}

The proof is inspired from a similar result that holds for C-adic words \cite{CLL17}. 
\begin{proof}
Let $w$ an Arnoux-Rauzy word; denote by $(s_n)_{n \in \N}$ its directive sequence and by $f$ its letter frequencies vector. We recall that for all nonnegative integer $n$, $s_n = \sigma_i$ if and only if the $i-th$ entry of $F_{AR}^n(f)$ ($F_{AR}$ is defined in Section \ref{sect:introduction}) is greater than the sum of the two others.  By contradiction, assume that the entries of $f$ are not rationally independent.

 First, observe that if for some $r \in \N$, the $i-th$ entry of $F^r_{AR}(f)$ is zero, then it will remain zero; and from this point on the directive sequence will not contain the substitution $\sigma_i$, which is conflicting with the definition of Arnoux-Rauzy words and the uniqueness of the directive sequence. Thus, for all $n \in \N$, all entries of $F_{AR}^n(f)$ are positive. Let $l_0$ a nonzero integer column vector such that $fl_0 = 0$ (recall that $f$ is a line vector). Let $l_m = M_{s_{m-1}}...M_{s_0}l_0$. The Arnoux-Rauzy matrices being invertible, $l_m$ is also a nonzero integer column vector; it satisfies $F_{AR}^m(f)l_m = f.M^{-1}_{s_0}...M^{-1}_{s_{m-1}}l_m = fl_0 = 0$.  
 Denote $l_m = (a,b,c)^t$ and consider $D_m = \max(\vert b-a \vert, \vert c-b \vert, \vert c-a \vert) \in \N$ the difference between the maximum and the minimum entry of $l_m$, that we call \emph{spread} of $l_m$. We claim that the sequence of nonnegative integers $(D_m)_{m \in \N}$ is non-increasing and that it furthermore decreases infinitely often - and here will be the contradiction.
 
 Indeed, the vector $l_{m+1}$ is of the form $Ml_m$, where $M$ is one the the three Arnoux-Rauzy matrices $M_{\sigma_1}$, $M_{\sigma_2}$ or $M_{\sigma_3}$, which give respectively: $l_{m+1} = (a,a+b,a+c)^t$, $l_{m+1}=(a+b,b,c+b)^t$ and $l_{m+1}=(a+c,b+c,c)^t$. One can easily show, observing that the extreme entries of $l_m$ have opposite signs, that in all cases $D_{m}\geq D_{m+1}$. Similarly, we write $l_{m+2}= M_{s_{m+1}}M_{s_m}l_m$. A quick argument show that as soon as $s_{m+1}\neq s_m$, which happens infinitely many times by definition of Arnoux-Rauzy words, we have $D_m>D_{m+2}$.
\end{proof}

%
%

\bibliographystyle{abbrv}
\bibliography{bib_Rauzy}

\end{document}